\documentclass[12pt]{amsart}
\usepackage{amsmath}
\numberwithin{equation}{section}
\usepackage{amsthm}
\usepackage{amssymb}
\usepackage{amsfonts}
\usepackage{pinlabel}
\usepackage{graphicx}
\newcommand{\nc}[2]{ \newcommand{#1}{#2} }

\nc{\avint}{ {- \hspace{-3.5mm} \int} }  

\nc{\R}{\mathrm{I \! R}}  
\nc{\N}{\mathrm{ I \! N}}  

\newcommand{\dclosure}[1]{ \stackrel{\rule{.2 in}{.01 in}}{#1} }

\newcommand{\pclosure}[1]{ \stackrel{\rule{.5 in}{.01 in}}{#1} }

\newcommand{\chisub}[1]{ {\mathbf{\chi}}_{_{#1}} }

\newcommand{\refeqn}[1]{ (\!\!~\ref{eq:#1}) } 
\newcommand{\refthm}[1]{ \!\!~\ref{#1} }    

\nc{\Holder}{H\"{o}lder\ }

\nc{\ith}{ \ensuremath{\text{i}^{\text{th}}} }
\nc{\jth}{ \ensuremath{\text{j}^{\text{th}}} }
\nc{\kth}{ \ensuremath{\text{k}^{\text{th}}} }
\nc{\dst}{ \ensuremath{\text{1}^{\text{st}}_{\delta}} }
\nc{\dnd}{ \ensuremath{\text{2}^{\text{nd}}_{\delta}} }
\nc{\ost}{ \ensuremath{\text{1}^{\text{st}}} }
\nc{\tnd}{ \ensuremath{\text{2}^{\text{nd}}} }
\nc{\curl}{ \nabla \times }
\nc{\Div}{ \nabla \cdot }
\nc{\DC}{K}

\newcommand{\BVPc}[4]{  
  \begin{equation}
        \begin{array}{rl}
           #1 & \ \text{in}
               \ \ #4 \vspace{.05in} \\
           #2 & \ \text{on} \ \ \partial #4 \;,
        \end{array}
  \label{eq:#3}
  \end{equation}    }

\newcommand{\BVPnA}[6]{  
  \begin{equation}
        \begin{array}{rll}
           #1 \!\!\!& = #2 & \ \text{in}
               \ \ #6 \vspace{.05in} \\
           #3 \!\!\!& = #4 & \ \text{on} \ \ \partial #6 \;
        \end{array}
  \label{eq:#5}
  \end{equation}    }

\nc{\Ppl}{ \mathcal{M}^{+} }  \nc{\Pmn}{ \mathcal{M}^{-} }

\nc{\smiley}{ $\stackrel{\because}{\smile} \;$ }

\theoremstyle{plain}
\newtheorem{theorem}{Theorem}[section]

\newtheorem{lemma}[theorem]{Lemma}
\newtheorem{corollary}[theorem]{Corollary}

\theoremstyle{definition}

\newtheorem{remark}[theorem]{Remark}

\def\qed{\hfill\rule{1ex}{1ex}\\}

\numberwithin{equation}{section}

\title[Mean Value Theorem]{Geometry of mean value sets for general divergence form uniformly elliptic operators}
\author[Aryal]{Ashok Aryal}
\author[Blank]{Ivan Blank}

\begin{document}
\baselineskip 18pt

\begin{abstract}
In the Fermi Lectures on the obstacle problem in 1998, Caffarelli gave a proof of the
mean value theorem which extends to general divergence form uniformly elliptic operators.
In the general setting, the result shows that for any such operator $L$ and at any point
$x_0$ in the domain, there exists a nested family of sets $\{ D_r(x_0) \}$ where the
average over any of those sets is related to the value of the function at $x_0.$  Although
it is known that the $\{ D_r(x_0) \}$ are nested and are comparable to balls in the sense that
there exists $c, C$ depending only on $L$ such that
$B_{cr}(x_0) \subset D_r(x_0) \subset B_{Cr}(x_0)$ for all $r > 0$ and $x_0$ in the domain,
otherwise their geometric and topological properties are largely unknown.  In this paper we
begin the study of these topics and we
prove a few results about the geometry of these sets and give a couple of applications of
the theorems.
\end{abstract}
\maketitle

\setcounter{section}{0}

\section{Introduction}   \label{Intro}

Based on the great importance of the mean value theorem in understanding harmonic
functions, it is clear that analogues for operators other than the Laplacian are automatically of interest.
In 1963, Littman Stampacchia, and Weinberger showed that if $\mu$
is a nonnegative measure on $\Omega$ and $u$ is the solution to
\BVPc{Lu = \mu}{u = 0}{LSWbvp}{\Omega}
and $G(x,y)$ is the Green's function for $L$ on $\Omega$
then $u(y)$ is equal to
\begin{equation}
    \lim_{a \rightarrow \infty}  \frac{1}{2a} \int_{a \leq G \leq 3a} u(x) a^{ij}(x) D_{x_i} G(x,y) D_{x_j} G(x,y) \; dx
\label{eq:LSWWMVT}
\end{equation}
almost everywhere, and this limit is nondecreasing  \cite[Equation 8.3]{LSW}.
On the other hand, this formula is not as nice as
the basic mean value formulas for Laplace's equation for a number of reasons.  First, it is an average
with weights, and not merely a simple average.  Indeed, the weights in question are not even easy
to estimate.  Second, it is not an average over a ball or something
which is even homeomorphic to a ball, but rather an average over level sets of the Green's function
which do not include the central point being estimated.

The following simpler mean value theorem was stated by Caffarelli in \cite{C2, C1} and proved carefully by
the second author and Hao within \cite{BH1}.  
\begin{theorem} [Mean Value Theorem for Divergence Form Elliptic PDE]  \label{BCHMVT}
Let $L$ be any divergence form elliptic operator with ellipticity $\lambda$, $\Lambda$.
For any $x_0 \in \Omega$, there exists an increasing family $D_R(x_0)$ which satisfies the following:
\begin{enumerate}
  \item $B_{cR}(x_0) \subset D_R(x_0) \subset B_{CR}(x_0),$ with $c,$ $C$ depending only on $n,$ $\lambda$ and $\Lambda$.
  \item For any $v$ satisfying $L v \geq 0$ and $R<S$, we have
       \begin{equation}
              v(x_0) \leq \frac{1}{|D_R(x_0)|}\int_{D_R(x_0)} v \leq  \frac{1}{|D_S(x_0)|}\int_{D_S(x_0)} v .
       \label{eq:MVTres}
       \end{equation}
\end{enumerate}
Finally, the sets $D_R(x_0)$ are noncontact sets of the following obstacle problem:  \\
$u \leq G(\cdot, x_0)$ such that
\BVPnA{L(u)}{- \chi_{\{u < G\}} R^{-n}}{u}{G(\cdot,x_0)}{KeyObProb}{B_M(x_0)}
where $B_M(x_0) \subset \R^n$ and $M >0$ is sufficiently large.
\end{theorem}
\noindent
Although this theorem has already been shown to be useful (see for example \cite{CR} as one place where it has already
been applied in this form), it is clear that the more that is known about the $D_R(x_0)$ the more useful the theorem is.  It
is also clear that although the fact that $B_{cR}(x_0) \subset D_R(x_0) \subset B_{CR}(x_0)$ for all $R$ gives us some
information about these sets, there is still much more that is unknown.

The present work actually originated as an attempt to better understand the solutions of a free boundary
problem of Bernoulli type.  In the celebrated paper of Alt and Caffarelli in 1981, nonnegative local minimizers of the
functional
\begin{equation}
    J(u) := \int_{D} (|\nabla u|^2 + \chisub{ \{ u > 0 \} }Q^2)
\label{eq:AltCaffFct}
\end{equation}
are studied \cite{AC}.  They are shown to exist and satisfy certain Lipschitz regularity estimates, and they obey
a linear nondegeneracy statement along their free boundary.  From there, Alt and Caffarelli turn to a study
of the free boundary.  This problem is also found (with $Q \equiv 1$) near the beginning of the text by
Caffarelli and Salsa \cite[Chapter 1]{CS}, and the first author of this paper was working on a generalization of that
problem for his dissertation.  In particular, we were considering the functional
\begin{equation}
    J_{a}(u) := \int_{D} (a^{ij} D_i u D_j u+ \chisub{ \{ u > 0 \} } )
\label{eq:GenAltCaffFct}
\end{equation}
with uniformly elliptic $a^{ij},$ and that will certainly color some aspects of the current work.
Unfortunately, after we started our project we learned of very nice and very recent
work of dos Prazeres and Teixeira which solved some of the problems that we had intended to publish
\cite{dPT}.  Nevertheless, their work had nothing to do with the MVT, and so we can now describe the
dual purpose of the current work:  First, we wish to state some theorems related to the
geometry of the $D_r(x_0).$  Second, we wish to show two applications in particular which illustrate
both the usefulness of the MVT, and the usefulness of our own results which give a more detailed
view of properties of the $D_r(x_0).$

The two biggest contributions that we make within this work regarding the properties of the $D_r(x_0)$
appear to be the following:
\begin{lemma}[Density Result]   \label{DensResPre}
Assume $y_0 \in \partial D_r(x_0),$ and assume that $c$ and $C$ are the constants given in
Theorem\refthm{gendivmvt}\!\!.  Fix $h \in (0,1/2).$  There exists a positive constant $\tau$
such that
\begin{equation}
     \frac{|B_{chr}(y_0) \cap D_r(x_0)|}{|B_{chr}(y_0)|} \geq \tau \;.
\label{eq:positivedensity}
\end{equation}
\end{lemma}
\noindent
This result prevents the $D_r(x_0)$ from having what might be described as an ``outward pointing cusp.''

\begin{lemma}[Continuous Expansion]   \label{ContExpPre}
Fix $x_0, y_0 \in \Omega$ and assume that there exists $0 < s < t$ so that
$y_0$ is not contained in $D_s(x_0),$ and is compactly contained within $D_t(x_0).$
Then there exists a unique $r \in (s,t)$ such that $y_0 \in \partial D_r(x_0).$
\end{lemma}
\noindent
This result allows us to state that the boundary of the mean value sets will move in a continuous fashion.

We were able to use the mean value theorem above in order to prove positive density of the contact set along the
free boundary.  Originally, we needed our two lemmas just mentioned in order to prove a nondegeneracy lemma
for the Bernoulli problem above.  Very recently, in joint work with Benson and LeCrone, the second author has
extended many of the results within this work to Riemannian manifolds \cite{BBL} in the case where $L$ is
the Laplace-Beltrami operator.  Indeed, all of the results from Section \ref{solidMVTdiv} can be extended to this case,
and when dealing with the obstacle problem on a compact Riemannian manifold $\mathcal{M}$ with boundary,
in order to be sure that the $D_r(x_0)$ can be extended until an $r_0$ where $\partial D_{r_0}(x_0)$
collides with $\partial \mathcal{M},$ we need the analogue of Lemma\refthm{ContExpPre}\!\!.  (See in particular
\cite[Corollary 4.9]{BBL}.)


\section{Solid MVT for divergence form elliptic operators}  \label{solidMVTdiv}


Let $\Omega$ be an open connected set in $\R^n,$ and let $A(x) = (a^{ij}(x))$ be a symmetric uniformly elliptic matrix.
That is for each $x \in \Omega$ we have unique
matrix $a^{ij}(x)$ satisfying:
\begin{equation}
    a^{ij} \equiv a^{ji} \ \ \ \ \text{(i.e. symmetry)}
\label{eq:Symm}
\end{equation}
and there exist $0<  \lambda \leq \mu < \infty$ such that 
\begin{equation}
    0<\lambda | \xi |^2 \leq a^{ij}(x)\xi_i\xi_j \leq \mu |\xi|^2 \ \ \text{for all} \ \ 
         \xi \in \R^n \setminus \{0\}, \ \text{and} \ \textit{x} \in \Omega,
\label{eq:UnifElliptic}
\end{equation}
which is called uniform ellipticity in this setting.  Although there are certainly very interesting operators
which are not uniformly elliptic, we will content ourselves to assume uniform ellipticity throughout this
entire work.


\begin{remark}[Analyst's Convention]  \label{AnalCon}
Notice that with our definition we can have $L = \Delta,$ but we won't have $L = - \Delta.$
\end{remark}

We consider the divergence form operator $ L:= \text{div}(A(x)\nabla(u)).$
For any $f \in L^{2}(\Omega),$ we will say that $u$ is a subsolution of $Lu = f$ (or more simply
$Lu \geq f$), whenever
$u\in W^{1,2}(\Omega)$ and for every $\phi \in W_{0}^{1,2}(\Omega), \ \phi \geq 0,$ we have
\begin{equation}
     -\int_{\Omega} a^{ij} D_i u D_j \phi \geq \int_{\Omega} f \phi \;.
\label{eq:subsoldef}
\end{equation}
Of course, supersolutions are defined in the same way, but with the inequality in
Equation\refeqn{subsoldef}reversed.

We recall here the main MVT that is the focus of our attention:
\begin{theorem}[MVT for divergence form elliptic PDE] \label{gendivmvt}
Let $L$ be a divergence form elliptic operator as described above.
For any $x_0 \in \Omega,$ there exist an increasing family $D_R(x_0)$ which satisfies the following:
\begin{enumerate}
     \item There exists $c$ and $C$ depending only on $n, \lambda,$ and $\mu,$ such that for all
$R > 0$ such that $B_{CR}(x_0) \subset \Omega$ we have
$B_{cR}(x_0) \subset D_R(x_0) \subset B_{CR}(x_0).$
     \item For any v satisfying $Lv \geq 0$ in $\Omega$ and any $0 < R< S,$ we have
\begin{equation}
        v(x_0) \leq \frac{1}{|D_R(x_0)|} \int_{D_R(x_0)}v 
                    \leq \frac{1}{|D_S(x_0)|} \int_{D_S(x_0)}v.
\label{eq:MVTDR}
\end{equation}
Finally, the sets $D_R(x_0)$ are noncontact sets of the following obstacle problem:  \\
$u \leq G(\cdot, x_0)$ such that
\BVPnA{L(u)}{- \chi_{\{u < G\}} R^{-n}}{u}{G(\cdot,x_0)}{KeyObProb2}{B_M(x_0)}
where $B_M(x_0) \subset \R^n$ and $M >0$ is sufficiently large.
\end{enumerate}
\end{theorem}


\begin{remark}[Dependencies] \label{depend}
It is shown in \cite{BH1} that for any $R > 0,$ the solution of the obstacle problem above becomes
independent of the choice of $M$ as long as it is sufficiently large, and we will always assume that
that is the case.  (It will be identically equal to the Green's function outside of the compact set $D_R(x_0).$)
We will frequently want to stress the dependence of the solution on $R,$ and so, accordingly, we will refer
to it as ``$u_R.$''  We will also use ``$w_R := G - u_R$'' when we wish to look at a function which, at least
away from $x_0$ satisfies the usual equations obeyed by the height function for an obstacle problem.
\end{remark}


\begin{remark}[Technicality]  \label{Technic}
Technically, we cannot use
the function $G(x,x_0)$ as boundary values in the sense of having a difference in $W_0^{1,2}$ until
we suitably remove the singularity at $x_0,$ so within \cite{BH1} they use a function that they call
$G_{sm}$ which agrees with $G$ within a neighborhood of the boundary but which has no singularity
in order to bypass this difficulty.
\end{remark}


The function $u_R$ is also the minimizer of
\begin{equation}
   J_R(u,\Omega):= \int_{\Omega} (a^{ij} D_i u D_j u - 2R^{-n}u)dx
\label{eq:JRdefn}
\end{equation}
among functions less than or equal to $G$ with boundary values equal to $G.$
Note that the Green's function $G$ of the general divergence form elliptic operator $L$ is the analogue of
the classical obstacle and $u_R$ is that of the membrane, and here the obstacle constrains the membrane
from above.

Although, as Caffarelli observed, the sets $D_R(x_0)$  are nested and comparable to balls in the sense that:
$$B_{cR}(x_0) \subset D_R(x_0) \subset B_{CR}(x_0) \;,$$
we know very little about the topology of the sets.  As a first small step in this direction
we offer the following lemma:
\begin{lemma}[Structure of $D_R(x_0)$]  \label{SimpStruct}
For any $x_0 \in \Omega$ and for any $R > 0$ such that $B_{CR}(x_0) \subset \Omega,$
the set $D_R(x_0)$ has exactly one component and it contains $x_0.$
\end{lemma}
\begin{proof}
Since $x_0 \in B_{cR}(x_0) \subset D_R(x_0),$ it is immediate that $x_0 \in D_r(x_0).$  Although
this statement is certainly trivial, we include it because of the observation that the MVT given
by Littman, Stampacchia, and Weinberger does not have this property.  (See\refeqn{LSWWMVT}above.)

Now for the next part, without loss of generality we can assume $x_0 = 0.$
Assume for the sake of obtaining a contradiciton that
$D_R(0)$ has a component that we will call $E$ which does not contain $0.$
Within $E$ we have $LG = 0, \ Lu_0 \leq 0,$ and $u_0 < G.$ On the other hand, it follows from
\cite{BH1} that $E$ is a bounded set, and since $u_0 = G$ on $\partial E,$ we contradict the weak
maximum principle.
\end{proof}


\begin{lemma}[Density Result]   \label{DensRes}
Assume $y_0 \in \partial D_r(x_0),$ and assume that $c$ and $C$ are the constants given in
Theorem\refthm{gendivmvt}\!\!.  Fix $h \in (0,1/2).$  There exists a positive constant $\tau$
such that
\begin{equation}
     \frac{|B_{chr}(y_0) \cap D_r(x_0)|}{|B_{chr}(y_0)|} \geq \tau \;.
\label{eq:positivedensity}
\end{equation}
Note that $\tau$ has no dependance on $x_0, y_0,$
or $r.$
\end{lemma}
\begin{figure}[!h]
	\centering
	\scalebox{.65}{\includegraphics{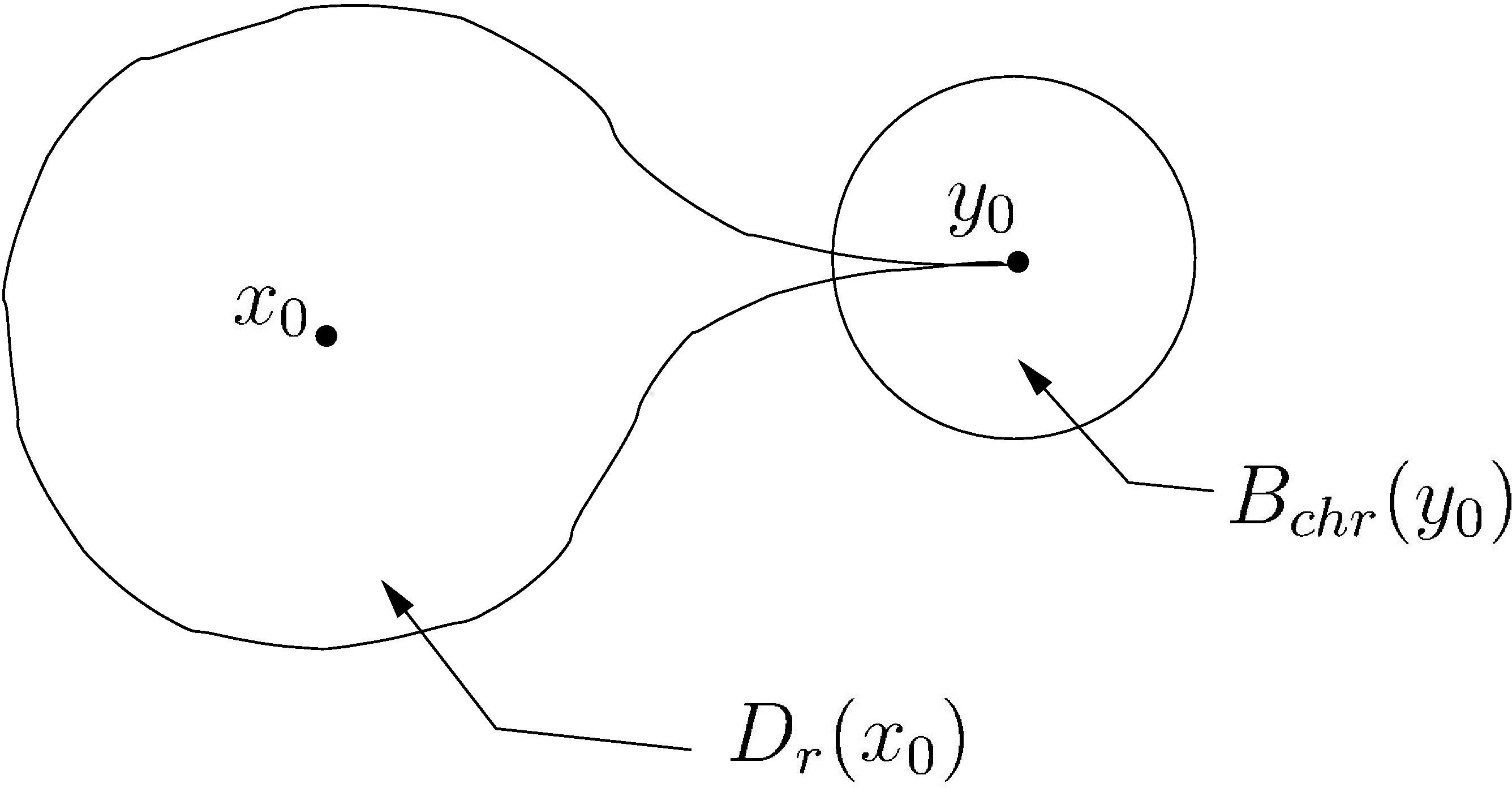}}
	\caption{Not possible according to the lemma.}
	\label{fig:lowdens}
\end{figure}
\begin{proof}
Without loss of generality we can rescale so that $r = 1.$
Observe that Theorem\refthm{gendivmvt}implies that that $x_0$ belongs to the complement
of $B_{ch}(y_0).$  It follows from the characterization of $D_1(x_0)$ as the noncontact set
for an obstacle problem along with the nondegeneracy result of the second author and Hao
(see Theorem 3.9 of \cite{BH1}) that there exists a point $z_0$ at a distance of $ch/2$ to $y_0$
where the solution to the corresponding obstacle problem has grown by an amount $\sim h^2.$
Next, by applying optimal regularity (see Theorem 3.2 of \cite{BH1}) we can be sure that there is
a ball with a radius bounded from below by a constant times $h$ which is centered at $z_0$ which
is not in the contact set.
\end{proof}


\begin{lemma}[Convergence of Minimizers]  \label{MinimConv}
For any $q > 0,$ we let $u_q$ minimize $J_q$ within the set:
\begin{equation}
   K_{M,G} := \{ u \in W^{1,2}(B_M) \; : \; u - G \in W_0^{1,2}(B_M), \ \text{and} \ u \leq G \; a.e. \}
\label{eq:KMGdef}
\end{equation}
where $J_q$ is as given in Equation\refeqn{JRdefn}above.
Now fix $r > 0.$  Then
\begin{equation}
    u_s \rightharpoonup u_r \ \text{in} \ W^{1,2}(B_M)
\label{eq:HappyLimit1}
\end{equation}
and
\begin{equation}
    \lim_{s \rightarrow r} ||u_s - u_r||_{C^{\alpha}(\dclosure{B_M})} = 0
\label{eq:HappyLimit2}
\end{equation}
for some $\alpha > 0.$
\end{lemma}
\begin{proof}
It is not hard to show that if $s_m$ is a sequence of positive numbers converging to $r,$ and if
we let $u_m := u_{s_m},$ then the sequence $\{u_m\}$ is uniformly bounded in $W^{1,2}(B_M)$
and uniformly bounded in $C^{\alpha}(\dclosure{B_M}).$  (See section 4 of \cite{BH1} for details.)
Thus, by using standard functional analysis we can be sure that there is a subsequence of $s_m$
which we will denote by $s_j$ such that we have
\begin{equation}
    u_j \rightharpoonup \tilde{u} \ \text{in} \ W^{1,2}(B_M) \ \ \text{and} \ \
    \lim_{j \rightarrow \infty} ||u_j - \tilde{u}||_{C^{\alpha}(\dclosure{B_M})} = 0
\label{eq:NYHL}
\end{equation}
for some function $\tilde{u} \in W^{1,2}(B_M) \cap C^{\alpha}(\dclosure{B_M}).$  Since the
original sequence $\{s_m\}$ was arbitrary, it remains only to show that $\tilde{u} = u_r.$

First note that for all of the $u_m$ in our sequence, we have:
\begin{equation}
     |J_r(u_m) - J_{s_m}(u_m)|  \leq \int_{B_M} \left| 2(s_m)^{-n} - 2r^{-n} \right| u_m
                                                \leq \left| 2(s_m)^{-n} - 2r^{-n} \right| \tilde{C}
\label{eq:SimpStupEst}
\end{equation}
where $\tilde{C}$ is the maximum of the $L^1$ norms of the $u_m.$  Of course, as we let $m \rightarrow \infty$
the right hand side goes to zero.  We know
\begin{alignat*}{2}
    J_{r}(u_r) &\leq J_r(\tilde{u}) \ \ \ \ \ & \ &\text{because} \ u_r \ \text{minimizes} \ J_r \\
                       &\leq \liminf_{j \rightarrow \infty} J_r(u_j) \ \ \ \ & \ &\text{by weak lower semicontinuity} \\
                       &= \liminf_{j \rightarrow \infty} J_{s_j}(u_j) \ \ & \ &\text{by using Equation\refeqn{SimpStupEst}\!\!.}
\end{alignat*}

Now we claim that
\begin{equation}
    \liminf_{j \rightarrow \infty} J_{s_j}(u_j) \leq J_r(u_r)
\label{eq:LPoC}
\end{equation}
which we can combine with the chain of inequalities in the previous paragraph along with uniqueness
of minimizers to show that $\tilde{u} = u_r.$ Suppose that this is not the case.  Then there exists
$s_k \rightarrow r$ and an $\epsilon > 0$ such that
\begin{equation}
    J_{s_k}(u_k) \geq J_r(u_r) + \epsilon \;.
\label{eq:DumbOrder}
\end{equation}
On the other hand, for sufficiently large $k,$ by using Equation\refeqn{SimpStupEst}again and then
Equation\refeqn{DumbOrder}we have
$$J_{s_k}(u_r) \leq J_r(u_r) + \epsilon/2 \leq J_{s_k}(u_k) - \epsilon/2 < J_{s_k}(u_k)$$
which contradicts the fact that $u_k$ is the minimizer of $J_{s_k}.$
\end{proof}


\begin{remark}[Statement for the $w_R$]   \label{WRS}
Of course in the language of the height functions $w_R,$ as long as $K$ is compactly contained
in the complement of $\{x_0\}$ we have
\begin{equation}
    \lim_{r \rightarrow s} ||w_r - w_s||_{C^{\alpha}(K)} = 0 \;.
\label{eq:wRconv}
\end{equation}
\end{remark}


%


\begin{lemma}[Continuous Expansion]   \label{ContExp}
Fix $x_0, y_0 \in \Omega$ and assume that there exists $0 < s < t$ so that
$y_0$ is not contained in $D_s(x_0),$ and is compactly contained within $D_t(x_0).$
Then there exists a unique $r \in (s,t)$ such that $y_0 \in \partial D_r(x_0).$
\end{lemma}
\begin{proof}
We borrow some of the ideas used in the proof of the counter-example within \cite{BT}.
Define the set of real numbers:
$$S := \{ \; t \in \R \; : \; y_0 \notin D_t(x_0) \; \} \;,$$
and let $r_0$ be the supremum of $S.$  Because the $D_r(x_0)$ are an increasing family of sets
with respect to $r,$ the set $S$ is an interval.  We claim that $y_0 \in \partial D_{r_0}(x_0).$
Assuming that $y_0 \notin \partial D_{r_0}(x_0),$ then there exists a $\rho > 0$ so that
\begin{equation}
    \text{dist}(y_0, \partial D_{r_0}(x_0)) = \rho \;.
\label{eq:ClosedSetDist}
\end{equation}
At this point there are two possible cases:  In the first case $B_{\rho}(y_0)\subset D_{r_0}(x_0),$
and in the second case $B_{\rho}(y_0)\subset D_{r_0}(x_0)^c.$

Suppose first that $B_{\rho}(y_0) \subset D_{r_0}(x_0).$  In this case, we have 
$\Pi := \pclosure{B_{\rho/2}(y_0)} \; \subset D_{r_0}(x_0) = \{ w_{r_0} > 0 \},$ and so
if
$$\gamma := \min_{\Pi} w_{r_0},$$
then $\gamma > 0.$  By Lemma\refthm{MinimConv}\!, there exists a sufficiently small $\delta > 0$
such that $|r - r_0| < \delta$ implies
\begin{equation}
    ||w_r - w_{r_0}||_{L^{\infty}(\Pi)} \leq \gamma/2 \;.
\label{eq:closeness1}
\end{equation}
Then the triangle inequality implies $w_r \geq \gamma/2 > 0$ in all of $\Pi \subset D_{r_0}(x_0)$
which contradicts the definition of $r_0.$

Next suppose that $B_{\rho}(y_0) \subset D_{r_0}(x_0)^c = \{ w_{r_0} = 0 \}.$  Within $B_{\rho}(y_0)$
the function $w_r$ satisfies the obstacle problem:
\begin{equation}
    Lw_r = \chi_{\{w_r>0\}} r^{-n}
\label{eq:ZeroObs}
\end{equation}
and therefore $w_r$ enjoys the quadratic nondegeneracy property.  (See section 3 of \cite{BH1}.)  Because
of this nondegeneracy, as long as $r > r_0,$ (and by using the definition of $r_0,$) we can guarantee that
there is a point within $\Pi := B_{\rho/2}(y_0)$ where $w_r$ is greater than a constant $\gamma > 0.$  On the
other hand, by Lemma\refthm{MinimConv}again, there exists a sufficiently small $\delta > 0$ such that
$|r - r_0| < \delta$ implies
\begin{equation}
    ||w_r - w_{r_0}||_{L^{\infty}(\Pi)} \leq \gamma/2 \;.
\label{eq:closeness2}
\end{equation}
Thus
$$0 < \gamma \leq ||w_r||_{L^{\infty}(\Pi)} = ||w_r - w_{r_0}||_{L^{\infty}(\Pi)} \leq \gamma/2$$
which gives us a contradiction for this case.  Hence we must have $y_0 \in \partial D_{r_0}(x_0).$                                                          
\end{proof}
\section{Applications to a Bernoulli-type free boundary problem}  \label{AppBTFBP}

We turn now to applications of the mean value results to the following problem:
Given $a^{ij}(x)$ as above, and boundary data, $\varphi \geq 0$ given on $\partial B_1,$
we consider minimizers of the functional:
$$ J_{a}(u) := \int_{B_1} (a^{ij} D_i u D_j u+ \chisub{ \{ u > 0 \} } ) $$
which we gave above in Equation\refeqn{GenAltCaffFct}for a general domain $D.$  Now in the case where
$a^{ij} \equiv \delta^{ij}$ the functional $J_a(u)$ simplifies to:
$$ J(u) := \int_{B_1} (|\nabla u|^2 + \chisub{ \{ u > 0 \} } ). $$
Alt and Caffarelli considered local minimizers of this functional,
and indeed this problem was used as a model
problem within the text by Caffarelli and Salsa.  We will say that $u_0$ is a local minimizer of
$J,$ if given any subdomain $D_0$ of $B_1$ the value of
$$J(u_0; D_0) := \int_{D_0} (|\nabla u_0|^2 + \chisub{ \{ u_0 > 0 \} } ), $$
is less than or equal to the value of $J(v;D_0)$ for any $v$ which is equal to $u_0$
on $\partial D_0.$

Some highlights of what is known about functions $u_0$ which locally minimize $J(u)$ in $B_1$ include the
following:
\begin{theorem}[Basic Facts for Minimizers of $J$]  \label{BasicFactsJ}
Within any $D_0 \subset \subset B_1$ we have:
\begin{enumerate}
    \item $u_0$ is Lipschitz.
    \item If $z_0 \in D_0 \cap \partial \{ u_0 > 0 \},$ then there is a constant $C > 0$ depending only
on $n$ and $||u_0||_{L^2(B_1)}$ such that
\begin{equation}
     C^{-1} r \leq \sup_{B_r(z_0)} u_0 \leq Cr \;.
\label{eq:abovenbelow}
\end{equation}
    \item With $z_0 \in D_0 \cap \partial \{ u_0 > 0 \}$ again, there is a universal $\theta > 0$ such that
\begin{equation}
    \mathcal{L}^n(\{ u_0 > 0 \} \cap B_r(z_0)) \geq \theta r^n \ \ \ \text{and} \ \ \ 
    \mathcal{L}^n(\{ u_0 = 0 \} \cap B_r(z_0)) \geq \theta r^n
\label{eq:DensRes}
\end{equation}
where we use $\mathcal{L}^n(S)$ to denote the n-dimensional Lebesgue measure of $S.$
    \item Using $\mathcal{H}^{\gamma}(S)$ to denote the $\gamma$-dimensional Hausdorff
measure of $S,$ then given $D_0 \subset \subset B_1$ there is a universal $C$ such that
\begin{equation}
     \mathcal{H}^{n-1}(\partial \{ u_0 > 0 \} \cap D_0) \leq C.
\label{eq:nmoHm}
\end{equation}
    \item $|\nabla u_0| = 1$ in a suitable sense on almost all of the free boundary.
\end{enumerate}
\end{theorem}
\noindent
Everything in the theorem above was proven by Alt and Caffarelli.  See \cite{AC,CS} for details.

More recently, dos Prazeres and Teixeira studied the local minimizers of the more general functional $J_a$ where
the $a^{ij}$ which appear are assumed to be no more than bounded, symmetric, and uniformly elliptic.  Now in
this case, there is no hope of proving that minimizers are better than the \Holder regularity given by the
famous result of De Giorgi and Nash.  On the other hand dos Prazeres and Teixeira proved that functions
$u_0$ which locally minimize $J_a(u)$ in $B_1$ satisfy the following:
\begin{theorem}[Basic Facts for Minimizers of $J_a$]  \label{BasicFactsJa}
Within any $D_0 \subset \subset B_1$ we have:
\begin{enumerate}
    \item If $z_0 \in D_0 \cap \partial \{ u_0 > 0 \},$ then there is a constant $C > 0$ depending only
on $n, \lambda, \Lambda,$ and $||u_0||_{L^2(B_1)}$ such that
\begin{equation}
     C^{-1} r \leq \sup_{B_r(z_0)} u_0 \leq Cr \;.
\label{eq:abovenbelowJa}
\end{equation}
     \item With $z_0 \in D_0 \cap \partial \{ u_0 > 0 \}$ again, there is a universal $\theta > 0$ such that
\begin{equation}
    \mathcal{L}^n(\{ u_0 > 0 \} \cap B_r(z_0)) \geq \theta r^n.
\label{eq:DensResJa}
\end{equation}
\end{enumerate}
\end{theorem}
See \cite[Theorem 1.1]{dPT}.
Also considered by dos Prazeres and Teixeira were $a^{ij}$ satisfying what they called the ``$K$-Lip''
property which do allow for Lipschitz estimates of the minimizers, but we never make this assumption.
(For those details, see \cite[Definition 3.3]{dPT}.)  Of course, even without any further hypotheses, one
can reasonably view Equation\refeqn{abovenbelowJa}as saying that ``at the free boundary'' the solutions
enjoy a Lipschitz-type behavior.  On the other hand, for general $a^{ij}$ one can easily construct a
counter-example to the statement: ``The one sided gradient exists at the free boundary''
by choosing $a^{ij}$ as in the paper by Blank and Teka, and then doing a blow up argument.  Thus, it
seems very difficult to get a successful analogue of the fifth statement in Theorem\refthm{BasicFactsJ}above.
It also seems difficult or impossible to prove Equation\refeqn{nmoHm}in the general case, although as
dos Prazeres and Teixeira observed, the free boundary is necessarily porous, and so if one is willing to
weaken $\mathcal{H}^{n-1}$ measure to $\mathcal{H}^{n- \zeta}$ measure for a $\zeta$ which is
between $0$ and $1,$ then one can assert the analogue \cite{dPT}.
From a certain point of view, the upshot is that the biggest gap between Theorem\refthm{BasicFactsJ}and
Theorem\refthm{BasicFactsJa}that we can hope to close is the fact that Equation\refeqn{DensResJa}is only
giving half of what Equation\refeqn{DensRes}gave, and that leads to our first application.

\subsection{Application 1: Positive Density of the Contact Set on the Free Boundary}

\begin{theorem}[Positive Density of the Contact Set on the Free Boundary]  \label{PDCSFB}
In the same setting as in Theorem\refthm{BasicFactsJa}and with $z_0 \in D_0 \cap \partial \{ u_0 > 0 \}$
there exists a $\theta > 0$ depending only on $n, \lambda, \Lambda,$ and $||u_0||_{L^2(B_1)}$ such that
\begin{equation}
    \mathcal{L}^n(\{ u_0 = 0 \} \cap B_r(z_0)) \geq \theta r^n.
\label{eq:OtherDensResJa}
\end{equation}
\end{theorem}

\begin{proof}
Let $v$ be a solution of the equation $Lu = 0$ in $B_r(x_0)$ with 
$v = u_0$ on $\partial B_r(x_0).$  Since $x_0$ is in the free boundary we know that $u_0$ and 
therefore $v$ is positive on a nontrivial portion of $\partial B_r(x_0).$  Then, the strong maximum
principle implies $v > 0$ in $B_r(x_0).$ 
Since $u_0$ is local minimizer we have,
$$\int_{B_r(x_0)}\left((A(x)\nabla u_0)\cdot\nabla u_0
         + \chi_{\{u_0 > 0\}}\right)
\leq \int_{B_r(x_0)}\left((A(x)\nabla v)\cdot\nabla v
         + \chi_{\{v > 0\}}\right) $$
which gives, 
\begin{alignat*}{1}
\int_{B_r(x_0)} \left( \rule[.033 in]{0in}{.12 in} (A(x)\nabla u_0) \cdot \nabla u_0 
                       - (A(x) \nabla v) \cdot \nabla v \right) 
   &\leq \int_{B_r(x_0)}\chi_{\{v > 0\}}- \int_{B_r(x_0)}\chi_{\{u_0 > 0\}} \\
   &= |B_r(x_0)| - |{\{u_0 > 0\}\cap B_r(x_0)}|\\
   &= | \{u_0 = 0\}\cap B_r(x_0) | \\
   &= | \Omega_0^c \cap B_r(x_0)| \;.
\end{alignat*}
On the other hand we claim that,
\begin{alignat*}{1}
\int_{B_r(x_0)} \left( \rule[.033 in]{0in}{.12 in} (A(x)\nabla u_0) \cdot \nabla u_0 
                     - ( A(x) \nabla v ) \cdot \nabla v \right)
   &= \int_{B_r(x_0)} \left( \rule[.033 in]{0in}{.12 in}  A(x) \nabla (u_0 - v) \right) \cdot \nabla (u_0 - v) \\ 
   &\geq \lambda \int_{B_r(x_0)} |\nabla (u_0 - v)|^2\\
   &\geq \frac{C \lambda}{r^2} \int_{B_r(x_0)} |( u_0 - v )|^2 \;.
\end{alignat*}
Thus, if we grant the claim, then we obviously have
\begin{equation}
       | \Omega_0^c \cap B_r(x_0) | \geq
        \frac{C\lambda}{r^2} \int_{B_r(x_0)}|(u_0 - v)|^2 \;.
\label{eq:L2boundbylabelset}   
\end{equation}

Turning to the proof of the claim we see immediately that
the last two inequalities in the chain of inequalities above
simply use uniform ellipticity and the Poincare inequality respectively.
Thus our claim is proved if we show the first equality.  So letting $\varphi := u_0 - v$ and 
observing that $\varphi \in W_0^{1,2}(B_r(x_0))$ we compute
\begin{alignat*}{1}
\int_{B_r(x_0)} (A(x)\nabla u_0) \cdot \nabla u_0 
   &- (A(x)\nabla v) \cdot \nabla v
    - (A(x)\nabla (u_0 - v))\cdot\nabla (u_0-v) \\ 
 &= 2\int_{B_r(x_0)} \left( \rule[.033 in]{0in}{.12 in} (A(x)\nabla u_0) \cdot \nabla v 
                - (A(x)\nabla v) \cdot \nabla v \right) \\
 &= 2\int_{B_r(x_0)}(A(x)\nabla v)\cdot \nabla(u_0 - v) \\
 &= 2\int_{B_r(x_0)}(A(x)\nabla v)\cdot \nabla(\varphi) \\
 &= 0
\end{alignat*}
\

\vspace{-.5in}

\noindent
since $Lv = 0$ in $B_{r}(x_0).$  Thus, the claim is proved.

Now using the MVT for general divergence form operators we get,
\begin{alignat*}{1}
 v(x_0)
         &= \frac{1}{|D_r(x_0)|}\int_{D_r(x_0)} v \\
         &\geq \frac{1}{|B_{Cr}(x_0)|} \int_{B_{cr}(x_0)} v \rule[.033 in]{0in}{.32 in} \\
         &=\frac{|B_{cr}(x_0)|}{|B_{Cr}(x_0)|} \rule[.033 in]{0in}{.32 in}
              \cdot \frac{1}{|B_{cr}(x_0)|} \int_{B_{cr}(x_0)} v \\
 &\geq \tilde{C} \frac{1}{|B_{cr}(x_0)|} \int_{B_{cr}(x_0)} u_0 \rule[.033 in]{0in}{.27 in}\\
 &\geq \tilde{C} r
\end{alignat*} 
where in the final inequality we have used both the nondegeneracy and the optimal regularity
of $u_0$ due to dos Prazeres and Teixeira \cite{dPT}.
Since $v$ is L-harmonic and nonnegative, the Harnack inequality tells us that
$v(y)\geq \tilde{C}r$ for all $y \in B_{r/2}(x_0).$ By the Lipschitz continuity of
$u_0$ we see that 
$u_0(y) \leq c_1 hr$ in $B_{hr}(x_0).$
By choosing $h$ to be sufficiently small we get
$$ v - u_0 \geq \hat{c}r \ \ \text{in} \ \ B_{hr}(x_0) \;.$$

Therefore by using Equation\refeqn{L2boundbylabelset}we get,
$$| \Omega_0^c \cap B_r(x_0) |
\geq \frac{c\lambda}{r^2}\int_{B_r(x_0)}|(u_0 - v)|^2
\geq \frac{c\lambda}{r^2}\int_{B_{hr}(x_0)} (\hat{c}r)^2
\geq Cr^n \;.$$
\end{proof}

By combining this last result with part (2) of Theorem\refthm{BasicFactsJa}we get the following
statement simply by definition.
\begin{corollary}[Measure Theoretic Boundary]  \label{FBisMB}
Every point of the free boundary belongs to the measure theoretic boundary of the zero set
and/or of the positivity set.
\end{corollary}
\noindent
Definitions and information about the measure theoretic boundary can be found in 
a variety of references on geometric measure theory including \cite{EG} and \cite{M}.

\subsection{Application 2: A Nondegeneracy Lemma}

\

Although the previous application of the MVT gives us a new result, it does not make use
of the new properties that we have shown.  On the other hand, by making use of our
lemmas in the second section, we can give a new proof of many of the results shown
independently by dos Prazeres and Teixeira.  Indeed, our method of proof follows the
exposition of Caffarelli and Salsa's text almost exactly, and so we will state here only
the proof of the key lemma that relies on our statements of the $D_r(x_0).$  This
lemma is the analogue of Lemma 1.10 of \cite{CS}.

\begin{lemma}[Nondegeneracy Lemma]   \label{eq:nondeglem}
Let $\Omega$ be an open set with $0 \in \partial\Omega$ and 
$w\geq 0,$ $||w||_{C^{0,1}(B_2)} = \bar{\mathcal{K}},$ and $Lw = 0$ in $\Omega\cap B_2.$
Suppose $x_0 \in \Omega \cap B_1$ and 
\begin{itemize}
    \item[(i)] $w(x_0) = \sigma >0,$ and
    \item[(ii)] in the region $\{ w \geq \sigma/3 \},$ 
we have $w(x)\sim \text{dist}(x,\partial\Omega).$
\end{itemize}
Then there exist positive constants $\eta, \beta, \gamma,$ and
$\sigma_0$ which all depend on $n, \lambda, \mu,$ and $\bar{\mathcal{K}},$
such that as long as $\sigma \leq \sigma_0,$ we have
\begin{equation}
     \beta \sigma \geq \sup_{B_{\eta \sigma}(x_0)} w \geq (1+\gamma) \sigma \;.
\label{eq:OneIt}
\end{equation}
\end{lemma}

\begin{proof}
Define $\rho > 0$ by
\begin{equation}
    \rho := \sup \{ r \in \R \; : \; \mathit{D_r(x_{\mathrm{0}}) \subset \{ w > \sigma/\mathrm{3} \} \} } \;,
\label{eq:rhodef}
\end{equation}
where $D_r(x_0)$ is the solid mean value set given in Theorem\refthm{gendivmvt}\!\!.  Using
Lemma\!~\refthm{ContExp} there exists a $y_0 \in \partial D_{\rho}(x_0)$ with $w(y_0) = \sigma/3.$
By assumptions (i) and (ii) we know that $\rho \sim \sigma.$
By the Lipschitz continuity of $w,$ for a suitable $h > 0,$ we have $w(x) \leq 2\sigma/3$ for all
$x \in B_{h\rho}(y_0).$  Now by using Lemma\refthm{DensRes}we know that $w \leq 2\sigma/3$
in a fixed proportion of $D_{\rho}(x_0).$  By the basic properties of the mean value sets $D_r(x_0),$
we have:
\begin{equation}
     \sigma = w(x_0) = \int_{D_{\rho}(x_0)} \mkern-50mu \rule[.033 in]{.095 in}{.01 in} \ \ \ \ \ w(y) dy \;,
\label{eq:MVThappiness}
\end{equation}
but since there is a fixed proportion of $D_{\rho}(x_0)$ where $w$ is less than $2\sigma/3$ we must have
a point in $D_{\rho}(x_0)$ which exceeds $\sigma$ by some fixed amount.  Since
$D_\rho(x_0) \subset B_{C\rho}(x_0)$ with $C$ as given in Theorem\refthm{gendivmvt}\!\!, and since as 
we observed above we have $\rho \sim \sigma,$ we get the
right hand side of Equation\refeqn{OneIt}\!\!.  The left hand side of Equation\refeqn{OneIt}follows
trivially from Lipschitz continuity so we are done.
\end{proof}

\noindent
Iterating this lemma in the same fashion that Caffarelli and Salsa iterate their Lemma 1.10 leads to the
key nondegeneracy theorem for solutions to this free boundary problem.

\bibliographystyle{plain}
\bibliography{GeomMVT}
\end{document}